\documentclass[11pt,thmsa]{article}%
\usepackage{amsmath}
\usepackage{amsfonts}
\usepackage{amssymb}
\usepackage{graphicx}%
\newtheorem{theorem}{Theorem}[section]

\newtheorem{definition}[theorem]{Definition}

\newtheorem{lemma}[theorem]{Lemma}

\newenvironment{proof}[1][Proof]{\noindent\textbf{#1.} }{\ \rule{0.5em}{0.5em}}

\begin{document}
\title{Clifford-Wolf translations of left invariant Randers metrics on compact Lie groups\footnote{Supported by NSFC(no.10671096 and 10971104) and SRFDP of China.}}
\author{  Shaoqiang Deng$^{1}$ and Ming Xu$^2$
\thanks{Corresponding author. E-mail: mgxu@math.tsinghua.edu.cn}\\
$^1$School of Mathematical Sciences and LPMC\\
Nankai University\\
Tianjin 300071, P. R. China\\
$^2$Department of Mathematical Sciences\\
Tsinghua University\\
Beijing 100084, P. R. China}
\date{}
\maketitle
\begin{abstract}
A Clifford-Wolf translation of a connected Finsler space  is an isometry which moves each point the sam distance. A Finsler space $(M, F)$ is called
Clifford-Wolf homogeneous if for any two point $x_1, x_2\in M$ there is a Clifford-Wolf translation $\rho$ such that $\rho(x_1)=x_2$.
In this paper, we study Clifford-Wolf translations of left invariant Randers metrics on compact Lie groups. The mian result is  that a left invariant Randers metric on a connected compact simple Lie group is Clifford-Wolf homogeneous if and only if the indicatrix of the metric is a round sphere with respect to a bi-invariant Riemannian metric. This  presents  a large number of examples of non-reversible Finsler metrics which are Clifford-Wolf homogeneous.


\end{abstract}

\section{Introduction}
Let $(M, g)$ be a homogeneous Riemannian manifold.
An isometry of the Riemannian manifold $(M, g)$ is called a Clifford-Wolf translation (or simply CW-transformation) if it moves each point of $M$ the same distance.  If for any two points $x_1, x_2\in M$, there is a CW-translation $\rho$ such that $\rho(x_1)=x_2$, then $(M, g)$ is called Clifford-Wolf homogeneous (or simply CW-homogeneous).

 The general merit of the study of CW-translations lies in the fact that they are closely related to the study of homogeneous Riemannian manifolds.
  Let $\Gamma$ be a properly discontinuous  subgroup of the full group of isometries of $(M, g)$, which acts freely  on $M$. It is a very important problem in Riemannina geometry to find out the conditions that the quotient manifold $M/\Gamma$ is again a homogeneous Riemannian manifolds. Through the work of a number of researchers, we can say that at least in the cases that $(M, g)$ is a Riemannian symmetric space, or a homogeneous Riemannian manifold of non-positive curvature, or a homogeneous Riemannian manifold admitting a transitive group of isometries which is semisimple of the noncompact type, $M/\Gamma$ is a homogeneous Riemannian manifold, provided that $\Gamma$ consists of Cw-translations. Moreover, in the first two cases it is proved that the above condition is also sufficient; see  \cite{WO60}, \cite{WO62}, \cite{WO64} and \cite{DMW86} for the details.

 In view of the the above results, it is an important problem to determine all the CW-transformations of a given Riemannian manifold, particularly for the homogeneous ones. For Riemannian symmetric spaces this problem was first completely settled by J. A. Wolf in \cite{WO62}; see also \cite{FR63, Oz69, Oz74} for other proofs of the results. The related results has important application in the study of homogeneous Riemannian manifolds of negative curvature; see for example \cite{HE74, AW76}.

Recently the study on CW-translation has become active. In \cite{BN081, BN082, BN09},  Berestovskii and  Nikonorov studied the local one-parameter group of CW-translations of general Riemannian manifolds. The main result is that there is a correspondence between local one-parameter groups of CW-translations and Killing vector fields of constant length. They also obtained a classification of connected simply connected CW-homogeneous Riemannian manifolds. The list consists of the euclidean spaces, the odd-dimensional spheres with constant curvature, connected simply connected compact simple Lie groups with bi-invariant Riemannian metrics and the direct products of the above manifolds. Notice that there are also a local version of the related notion of CW-homogeneous Riemannian manifolds; see the above cited papers above.

In our recent papers \cite{DP} and  \cite{DM12}, we initiated the study of CW-translations of Finsler spaces. The notions of CW-translations and CW-homogeneity have been generalized to the Finslerian setting. It is proved that the correspondence between the local one-parameter subgroup of CW-translations and the Killing vector fields of constant length is also valid for a Finsler space; see the next section for the details. We also studied this problem for a special type of Finsler spaces-Randers spaces and obtained a necessary and sufficient condition for a smooth vector field $X$ on a homogeneous Randers space $(G/H, F)$, with $H$ connected, to be a Killing vector field of constant length. This result gives a complete classification for all the local one-parameter subgroup of CW-translations of a homogeneous Randers space.

In this paper we continue our study on this topic. Our main idea is to study the condition for a homogeneous Randers space to be CW-homogeneous, and futhermore, to classify all the CW-homogeneous Randers spaces. However, due to the complexity of the problem, we will only deal with left invariant Randers metrics on a compact simple Lie group. The main result of this paper is the following

\begin{theorem}\label{thmmain}
Let $G$ be a compact connected simple Lie group and $F$ be a left
invariant Randers metric. Then $(G,F)$ is CW-homogeneous if and only
if the indicatrix of $F$ in $\mathfrak{g}$ is a round sphere with respect to a
bi-invariant metric.
\end{theorem}

 Recall that the indicatrix of a Finsler space $(M, F)$ at a point $x\in M$ is defined to be
 $$\mathcal{I}_x=\{y\in T_x(M)|F(x, y)=1\}.$$
 Notice that in the above theorem, we have identified $\mathfrak{g}$ with $T_eG$. On the other hand,
since $G$ is simple, up to a positive scalar the  bi-invariant Riemannian metric on $G$ is unique. Therefore the statement of the theorem have no independence on the specific bi-invariant metric. The above theorem says that a left invariant Randers metric on a compact simple is CW-homogeneous if and only if it solves the Zermelo's navigation problem of a bi-invariant Riemannian metric under the influence of certain external vector field. Put another way, this theorem presents many examples of non-reversible Finsler spaces which are CW-homogeneous.  It would be an interesting problem to classify CW-homogeneous Finsler spaces, either reversible or non-reversible.

The arrangement of this paper is as the following. In Section 2, we recall some notions and known results on Finsler geometry and CW-translations of Finsler spaces. In Section 3, we study Killing vector fields of left invariant Randers metrics and obtain a complete description of those with constant length. In Section 4, we first state a theorem (Theorem \ref{preparation-theorem}) and then use this theorem to give a proof of Theorem \ref{thmmain}. Finally, in Section 5, we give a case by case proof of Theorem \ref{preparation-theorem}.

\section{Preliminary}
Finsler geometry is introduced by Riemann in his celebrated lecture
on the foundations of geometry,  addressed in 1854. Due to the complexity of the problem, the study of Finsler geometry was dormant for a rather long period.  In 1918,   Finsler studied the variation problem of Finsler spaces in his
doctoral dissertation, which  initiated the systematic study of
Finsler geometry. For fundamental properties of Finsler spaces, we refer the readers to \cite{BCS00}, \cite{CS04} and \cite{SH2}.

\begin{definition}
A Finsler metric on a manifold $M$ is a continuous function
$F:TM\rightarrow \mathbb{R}^+$, which is smooth on the slit
tangent bundle $TM\backslash 0$. In any local standard coordinates
$(x^i,y^j)$ for $TM$, $F$
satisfies the following conditions:

{\rm (1)} $F(x,y)>0$ for any $y\neq 0$.

{\rm (2)} $F(x,\lambda y) =\lambda F(x,y)$ for any $y\in TM_x$, and
$\lambda >0$.

{\rm (3)} The Hessian matrix defined by $g_{ij}=\frac{1}{2}[F^2]_{y^i
y^j}$ is positive definite.
\end{definition}

An important example  is  a Finsler metric with the form $F=\alpha+\beta$, where
$\alpha$ is a Riemannian metric and $\beta$ is a 1-form. This kind of Finsler metrics are called
Randers metrics. Notice that the norm of the $1$-form $\beta$ with respect to the
metric $\alpha$ must  be smaller than $1$ and they reduce to
Riemannian metrics if and only if the $\beta$-terms vanish. Randers metrics are among the
most important examples of non-Riemannian Finsler metrics in the
study of geometry and physics.

Using the integration of $F$ along a path, we can define the arc length of piece-wise smooth curves.  The ``distance" of two points is then
defined to be the infimum of the arc length of all the piece-wise
smooth curves connecting them.  We  call it the distance
function and denote  as $d(\cdot,\cdot)$. In general, the distance function of a Finsler space is not reversible (i.e., $d(x,y)=d(y,x)$ for any $x,y$)
 unless $F(x,y)=F(x,-y)$,  for all $x\in M$ and
$y\in TM_x$.

The notion of CW-translations can be generalized to Finsler geometry; see
\cite{DM12}.

\begin{definition}
A  CW-translation  $\rho$ of
a Finsler manifold $(M,F)$ is an isometry of $(M,F)$ such that
$d(x,\rho(x))$ is a constant function.
\end{definition}

In the Riemannian case, the interrelation between CW-translations
and Killing vector fields of constant lengths is studied by
V.N.Berestovskii and Yu.G.Nikonorov. The key observation is that the
flow curves of a Killing vector field of constant length are
geodesics, so Killing vector fields of constant length can be used
to generate  one-parameter groups of CW-tranlations  when the parameter var1able is
close to $0$. This observation is still valid in the Finslerian case,
except that the Killing vector fields of constant lengths can only
generate CW-translations with positive small parameter variables
when $F$ is not absolutely homogeneous. In fact, we have the following
interrelation theorems (\cite{DM12}).

\begin{theorem}
Let $(M,F)$ be a complete Finsler manifold with a positive
injective radius. If $X$ is a Killing vector field of constant
length, then the flow $\phi_t$ generated by $X$ is a CW-translation
for all sufficiently small $t>0$.
\end{theorem}

\begin{theorem}
Let $(M,F)$ be a compact Finsler manifold. Then there is a
$\delta>0$, such that any CW-translation $\rho$ with
$d(x,\rho(x))<\delta$ is generated by a Killing vector field of
constant length.
\end{theorem}

We have also generalized the concept of (restrictively) CW-homogeneous space in Finsler geometry.

\begin{definition}
A Finsler manifold $(M,F)$ is called  CW-homogeneous  if for any two points $x_1,x_2\in M$, there
is a CW-translation $\rho$ such that $\rho(x_1)=x_2$. It is called
restrictively
CW-homogeneous  if for any point $x\in M$, there is a neighborhood
$V$ of $x$, such that for any two points $x_1,x_2\in V$ there is a
CW-translation $\rho$ mapping $x_1$ to $x_2$.
\end{definition}

Notice that the set of CW-translations of a Finsler space is generally not a subgroup of the full group of isometries. Therefore, restrictive CW-homogeneity is a weaker notion than CW-homogeneity.
In the case that $(M,F)$ is a compact Finsler manifold,  the interrelation
between CW-translations and Killing vector fields of constant
length gives an equivalent description of the restrictive
CW-homogeneity, namely, the Finsler space $(M,F)$ is restrictively CW-homogeneous if and only if any
non-vanishing tangent vector can be extended to a Killing vector
field of constant length.

\section{Left invariant Randers metrics  and Killing vector fields of
constant length}

Let $G$ be a connect compact Lie group with Lie algebra
$\mathfrak{g}$. A Finsler metric $F$ on $G$ is called left invariant
if the left translation group $L(G)$ is contained in the isometry
group of $F$. A left invariant Finsler metric is completely   determined by its restriction to
$T_eG\cong \mathfrak{g}$. A left invariant Randers $F=\alpha+\beta$
is then determined by an inner product on $\mathfrak{g}$ and an element of
 $\mathfrak{g}^*$,  which define the $\alpha$ and $\beta$ terms
respectively. Let $\langle\,,\,\rangle$ be the inner product induced by
$\alpha$ and $\langle\,,\,\rangle_{\mathrm{bi}}$ be a fixed $\mathrm{Ad}(G)$-invariant inner
product on $\mathfrak{g}$ induced by a bi-invariant metric on $G$,
then the restriction of $F$ to $T_eG$ can be written as
$F(y)=\sqrt{\langle y,y\rangle}+\langle y,V\rangle_{\mathrm{bi}}$ for some $V\in \mathfrak{g}$.

The Killing vector fields with constant length of a left invariant
Randers metric is the key subject for our considerations. In
\cite{DM12}, we have already discussed some techniques for determining
Killing vector fields $X\in\mathfrak{g}$ with constant length, and found some examples in which the
CW-translations are different from those Riemannian ones.
Any Killing vector field $X$ belongs to the Lie algebra of the full
isometry group of $F$, which contains $\mathfrak{g}$ as a Lie
subalgebra corresponding to the Lie subgroup $L(G)$. Therefore it is important to compute the full isometries of $(G, F)$. In general,  it is rather difficult to get a complete
classification of all the isometry groups of left invariant Randers metrics on a compact Lie group. However, the following
lemma tells us that its identity component  is contained in $L(G)R(G)$ when $G$ is simple (i.e.,
$\mathfrak{g}$ is simple).

\begin{lemma}
Let $G$ be a compact connected simple Lie group and $F$  a left
invariant Randers metric on $G$. Then the identity component $I_0(G,F)$ of the full group $I(M, F)$ of isometries of $(G, F)$ is contained in
$L(G)R(G)$.
\end{lemma}

\begin{proof}
Let $F=\alpha+\beta$ be the standard decomposition for the left
invariant Randers metric. Then we have $L(G)\subset I(G,F)\subset
I(G,\alpha)$, i.e.,  $\alpha$ is a left invariant Riemannian metric. Since $G$
is compact connected simple, $I_0(G,\alpha)$ is contained in
$L(G)R(G)$ (\cite{OT76}). Thus  $I_0(G,F)\subset I_0(G,\alpha)$ is also
contained in $L(G)R(G)$.
\end{proof}

Let $G'$ be the subgroup of $G$ such that $R(G')$ is the connected
component of the group of all isometric right translations, and
denote $\mathrm{Lie}(G')=\mathfrak{g}'$. Then obviously $I_0(G,F)=L(G)R(G')$.
In the Lie algebra level, the Lie algebra of $I_0(G,F)$, i.e., the
space of all Killing vector fields, is the direct sum of
$\mathfrak{g}$ and $\mathfrak{g}'$.

The following theorem gives a criterion  for  a Killing vector field
$(X,X')\in\mathfrak{g}\oplus\mathfrak{g}'$ of the left invariant
Randers metric $F$ to have  constant length.

\begin{theorem}\label{main-theorem-1}
If $(X,X')\in\mathfrak{g}\oplus\mathfrak{g}'$ generates a Killing vector
field of constant length of a left invariant Randers metric $F$,
then either $X=0$ or $X'$ belongs to the center
$\mathfrak{c}(\mathfrak{g}')$ of $\mathfrak{g}'$.
\end{theorem}

\begin{proof}
Let $F(y)=\sqrt{\langle y, y\rangle}+\langle y,V\rangle_{\mathrm{bi}}$ be the restriction of $F$ to
$T_eG$. If $(X,X')\in\mathfrak{g}\oplus\mathfrak{g}'$ generates a
Killing vector field of constant length for $F$, then the projection
of the $\mathrm{Ad}(L(G)R(G'))$-orbit of $(X,X')$ in $\mathfrak{g}$ has the
same $F$ values, i.e.,
\begin{eqnarray}
F(\mathrm{Ad}((g,g'))(X,X'))&=&\alpha(\mathrm{Ad}(g) X - \mathrm{Ad}(g') X') + \langle \mathrm{Ad}(g) X -
\mathrm{Ad}(g') X',V\rangle _{\mathrm{bi}}\nonumber\\
&\equiv& \mbox{constant}, \forall g\in G, g'\in
G'.
\end{eqnarray}
Since $\langle \mathrm{Ad}(g')X',V\rangle _{\mathrm{bi}}=\langle X',V\rangle$, we see that for any fixed $g\in
G$, $\alpha(\mathrm{Ad}(g) X - \mathrm{Ad}(g') X')$ is a constant function of $g'\in
G'$. So
\begin{eqnarray}
\langle \mathrm{Ad}(g) X, \mathrm{Ad}(g') X'\rangle _{\alpha}&=&\frac{1}{2}(\alpha(\mathrm{Ad}(g)
X)+\alpha(\mathrm{Ad}(g')X')-\alpha(\mathrm{Ad}(g) X-\mathrm{Ad}(g') X')) \nonumber \\
&=&\frac{1}{2}(\alpha(\mathrm{Ad}(g) X)+ \alpha(X')-\alpha(\mathrm{Ad}(g) X- \mathrm{Ad}(g')
X'))
\end{eqnarray}
is a constant function of $g'\in G'$, for any fixed $g\in G$. Now Select
$$g'=\exp(t_1 Y_1)\cdots\exp(t_n Y_n).$$
  Taking the derivative with respect to   all the $t_i$'s
and evaluating at $0$, we  easily deduce that for any
$g\in G$, the vector $\mathrm{Ad}(g) X$ is orthogonal to the ideal generated
by $[X',\mathfrak{g}']$ in $\mathfrak{g}'$ with respect to  $\alpha$.
Since the above assertion holds for any  $g\in G$, we  see that the
ideal generated by $X$ in $\mathfrak{g}$ is orthogonal to the ideal
generated by $[X',\mathfrak{g}']$ in $\mathfrak{g}'$. Since
$\mathfrak{g}$ is simple, this implies that either  $X=0$ generates
the $0$ ideal, or $X\neq 0$ and it generates the ideal $\mathfrak{g}$. Notice that in the later case we have
 $[X',\mathfrak{g}']=0$, i.e.,
$X'\in\mathfrak{c}(\mathfrak{g'})$. This completes the proof of the theorem.
\end{proof}

%

Using a similar argument we can prove the following modification of Theorem
\ref{main-theorem-1}, which allows more general $G'$.

\begin{theorem}\label{main-theorem-2}
Let $G$ be a compact connected simple Lie group, and $G'$ a closed
connected subgroup of $G$ with Lie algebra $\mathfrak{g}'$. Let $F$
be a $L(G)R(G')$-invariant Randers metric on $G$. If
$(X,X')\in\mathfrak{g}\oplus\mathfrak{g}'$ generates  a Killing vector
field of constant length, then either $X=0$ or $X'$ belongs to the
center $\mathfrak{c}(\mathfrak{g}')$ of $\mathfrak{g}'$.
\end{theorem}

In Theorem \ref{main-theorem-2}, we treat $(G,F)$ as the homogeneous
Randers space, in which is viewed as the coset space $G=(L(G)R(G'))/H$ and $H$ can be identified
with the quotient of $G'$ by a discrete normal subgroup. The Lie
algebra $\mathfrak{h}$ of $H$ is isomorphic to $\mathfrak{g}'$ and
it is a ``diagonal'' in $\mathfrak{g}\oplus\mathfrak{g}'$. The
projection from $(X,X')\in\mathfrak{g}\oplus\mathfrak{g}'$ to
$\mathfrak{m}\cong\mathfrak{g}$ is just $X-X'$.

For simplicity, we assume that $(X,X')\in\mathfrak{g}\oplus
\mathfrak{c}(\mathfrak{g}')$, with $X\neq 0$, generates a Killing vector
field of constant length $1$ for the $L(G)R(G')$-invariant Randers
metric $F$ on $G$. Then the projection from the orbit to
$\mathfrak{g}$ is just a shift by $-X'$. Its projection in
$\mathfrak{g}$ is contained in an indicatrix ellipsoid $S$ for $F$
if and only if the $\mathrm{Ad}(G)$-orbit of $X$ is contained in another
ellipsoid $S'$,  which is the indicatrix for another left invariant
Randers metric $F'$. Then $X\in\mathfrak{g}$ generates  a Killing vector
field of constant length $1$ of $F'$, and $X'$ must be inside the
indicatrix ellipsoid of $F'$. The Randers metric $F$ is invariant
under the right multiplications in $G'$ if and only its indicatrix $S$
in $\mathfrak{g}$ is $\mathrm{Ad}(G')$-invariant. By Theorem
\ref{main-theorem-2}, the shifting from $S$ to $S'$ by $X'$ is
$\mathrm{Ad}(G')$-invariant, so $F'$ is also $L(G)R(G')$-invariant. The
correspondence from $F'$ to $F$ is similar. We have thus proved the
following theorem.

\begin{theorem}\label{main-theorem-3}
Let $G$ be a compact Lie group with simple Lie algebra
$\mathfrak{g}$, $G'$ a closed connected subgroup of $G$ with Lie
algebra $\mathfrak{g}'$. Then for any
$(X,X')\in\mathfrak{g}\oplus\mathfrak{c}(\mathfrak{g}')$, $X\neq 0$,
and $l>0$, there is an one-to-one correspondence between the
following two sets:
\begin{description}
\item{\rm (1)}\quad The set of all $L(G)R(G')$-invariant metrics such that $(X,X')$
generates a Killing vector field of constant length $l$.
\item{\rm (2)}\quad  The set of all $L(G)R(G')$-invariant metrics such that $X$ generates
a Killing vector field of constant length $l$ and the length of $X'$
with respect to  the Randers metric is less than $l$.
\end{description}
\end{theorem}

Theorem \ref{main-theorem-3} provides a theoretical  machinery to
find left invariant Randers metrics with von-vanishing Killing
vector fields of constant length. First we fix a nonzero $X\in \mathfrak{g}$ and find the
left invariant Randers metrics such that $X$ generates a Killing vector
field of constant length for this metric (as we did in the previous work
\cite{DM12}). If its isometry group is the product of $L(G)R(G')$ as
given above, then we can freely choose any $X'$ from
$\mathfrak{c}(\mathfrak{g}')$ which is shorter than $X$. In this way we
get the a Randers metric by requiring  its indicatrix to be  the parallel
shifting of the former one by $-X'$. Then it is easily seen that the above Randers
metric has  a Killing vector field of constant length generated by
$(X,X')$.

Besides the Killing vector fields generated by $(X,X')\in\mathfrak{g}\oplus
\mathfrak{c}(\mathfrak{g})$ with $X\neq 0$, there are Killing vector
fields of constant lengths generated by elements of  the form $(0,X')$. It is easy to see that in this case
$X'$ can be any vector in $\mathfrak{g}'$. Since the subgroup $R(G')\in
I(G,F)$ commute with $L(G)$, and $L(G)$ acts transitively on $G$,
$R(G')$ gives a group of CW-translations which is neither new nor
interesting for our consideration (see \cite{DM12}).

\section{Left invariant CW-homogeneous Randers metrics on simple compact Lie groups}
We now apply  Theorem \ref{main-theorem-3} to give a proof of  Theorem \ref{thmmain}.
The proof contains two steps. We first  prove the statement of the theorem with
CW-homogeneity replaced by restrictive CW-homogeneity. Then we
prove that, for a left invariant Randers metric on a connected simply connected compact simple Lie group, the  CW-homogeneity and restrictive homogeneity are equivalent.

The first step of the proof needs the following theorem, which will
be proved in the next section.

\begin{theorem}\label{preparation-theorem}
Let $G$ be a compact connected simple Lie group with Lie algebra
$\mathfrak{g}$. Then a generic $X\in\mathfrak{g}$ can not generate a
Killing vector field of constant length when the  left
invariant Randers metric $F$ on $G$ is not Riemannian. The
complement of all those generic elements  is a subvariety with a
codimension at least $\mathrm{rk}(\mathfrak{g})+1$.
\end{theorem}

We remark here the generic condition is some condition for the eigenvalue multiplicities when $\mathfrak{g}$ is realized as a Lie algebra of matrices, and the exact sense may depend on the explicit Lie algebras; see the interruption in Section 5 for the Lie algebras of $A_n$, $\geq 1$, $D_n$ with $n$ odd and $>2$, and $E_6$. Hence apparently the generic condition has not been precisely defined. However, what is important in the above theorem is the assertion on the codimension of the set of generic elements. In fact, in the proof of Theorem \ref{preparation-theorem}, we will only use the generic condition in the above mentioned three cases, in which the condition can be precisely described.

\medskip
\noindent\textbf{Proof of Theorem \ref{thmmain}.}\quad
Let $I_0(G,F)$ be the product $L(G)R(G')$, with $G'\subset G$ and
$R(G')$ the connected subgroup of all right translation isometries.
Denote the Lie algebra of  $I_0(G,F)$ by
$\mathfrak{g}\oplus\mathfrak{g'}$.
If the indicatrix of the $L(G)R(G')$-invariant metric $F$ in
$T_eG \cong\mathfrak{g}$ is a round sphere,
centered at $-X'$, with radius $r>0$ with respect to  the bi-invariant metric, then $X'$ is
$\mathrm{Ad}(G')$-invariant, i.e., $X'$ lies in $ \mathfrak{c}(\mathfrak{g})$. Since the
indicatrix of $F$ contains $0$,  the length of $X'$ with respect to
the bi-invariant metric satisfies $|X'|_{\mathrm{bi}}<r$. Any vector
$(X,X')$, with $|X|_{bi}=r$, generates a Killing vector field of
constant length, and  these vectors exhaust all tangent directions in $T_eG$.
Through left translations, the Killing vector fields of constant
lengths can exhaust all tangent directions at any point. Thus $(G,F)$
is restrictively CW-homogeneous.

 Conversely,
if $(G,F)$ is restrictively CW-homogeneous, then any tangent vector
$X''\in \mathfrak{g}\cong T_eG$, of  length $1$ with respect to $F$ can be
extended to a Killing vector field of constant length $1$. Such  a
Killing vector field $(X,X')$ has either of the following two forms:
\begin{description}
\item{(1)}\quad $X\in\mathfrak{g}$, $X'\in\mathfrak{c}(\mathfrak{g}')$, and
$X=X'+X''$;
\item{\rm (2)}\quad $(0,X')$, with $X'=-X''\in\mathfrak{g'}$.
\end{description}
 Notice
that if $F$ is Riemannian, then it must be a bi-invariant metric
(\cite{DM12}). Now suppose $F$ is not Riemannian and write $F$ as
$F(y)=\sqrt{\langle y,y\rangle}+\langle y,V \rangle_{\mathrm{bi}}$, where $y\in T_eG$ and
$V$ is a non-vanishing vector in $\mathfrak{g}$.  Then $[X',V]=[X'',V]=0$ for a Killing vector
field $(X', X'')$ of constant length with the second form. In this case $X''$ is
contained in a subspace with a lower dimension. So when $X''$ is
generic, the corresponding Killing vector field $(X,X')$ is of the
first form, i.e.,  $X'\in\mathfrak{c}(\mathfrak{g}')$. Theorem
\ref{main-theorem-3} implies that $X$ generates a Killing vector field of
constant length for another $L(G)R(G')$-invariant Randers metric
$F'$, whose indicatrix is just a shift of that of $F$. Notice that
all the possible $X'$'s are contained in $\mathfrak{c}(\mathfrak{g}')$
which has a dimension at most $\mathrm{rk}(\mathfrak{g})$. For generic $X''$,
the corresponding $X=X'+X''$ must be generic in the sense of Theorem
\ref{preparation-theorem}, otherwise $X$ belongs to a subvariety
with a codimension at least $rk(\mathfrak{g})+1$, and this implies that
$X''=X-X'$ belongs to subvariety with a codimension at least $1$,
which conflict with the assumption that $X''$ is generic. So the corresponding $F'$
for $X$ must be Riemannian, or equivalently, its indicatrix must be
the only shifting of that of $F$ with its center shifted back to $0$. Thus for a
generic vector $X''$ with  length $1$ for $F$, the corresponding
Killing vector field $(X,X')$ of constant length  has the same $X'$
factor, and the corresponding $F'$ for the $X$ term is also the
same. As the set of   the Killing vector fields with a fixed constant length is a closed subset of the set of all the Killing vector fields, the above assertion is true for all $X''$'s with  length $1$ for
$F$. So $F'$ is a restrictively CW-homogeneous Riemannian metric, i.e.,
it is a bi-invariant Riemannian metric. Consequently the indicatrix of $F'$ in $\mathfrak{g}$ is a round
sphere. Hence  its shifting, the indicatrix of $F$,  is also a round
sphere.

Up to now we have completed the first step of the proof. For the second step we only need to  prove that for a left invariant Randers metric on a connected simply connected simple Lie group, the restrictive CW-homogeneity  implies the
CW-homogeneity.

By suitable scalar changes, we can assume that $\{(X,V)| |X|_{\mathrm{bi}}=1\}$
generates all the Killing vector fields of constant length $1$ with respect to $F$, in
which the fixed $V$ satisfies $|V|_{\mathrm{bi}}<1$. Then after a constant
re-scaling of the parameter, any geodesic of $(G,F)$ starting from
$g$ can be written as $\exp (tX) g \exp (-tV)$ with $t\geq 0$, for
some $X\in\mathfrak{g}$ with $|X|_{\mathrm{bi}}=1$. The  geodesic $\exp (tX) g
\exp (-tV)$, $t\in [0,t_0]$ from $g$ to $g'$ is not minimizing if
and only if there is another geodesic from $g$ to $g'$ with the form
$\exp (tX') g \exp (-tV)$, $|X'|_{\mathrm{bi}}=1$ and $t\in[0,t']$ with
$t'<t_0$. This implies that
\begin{equation}\label{geo}
\exp(t_0 X) g \exp (-t_0 V) = \exp (t' X') g \exp (- t' V),
\end{equation}
i.e., $\exp(t_0 X) = \exp (t' X') \exp((t_0-t')\mathrm{Ad}(g) V)$. For the
bi-invariant metric, the geodesic $\exp (tX)$, $t\in [0,t_0]$ from
the unit element $e$ to $g'g^{-1}$ has a length $t_0$. The right
side of (\ref{geo}) gives a path from  $e$  to $\exp (t' X') \exp((t_0-t')\mathrm{Ad} (g)
V)$, which is a geodesic from
$e$ to $\exp((t_0-t') V)$ with length $(t_0-t')|V|_{bi}$,  and a
geodesic from $\exp((t_0-t')V)$ to $\exp (t' X') \exp((t_0-t')\mathrm{Ad}(g)
V)$ with length $t'$. The total length is less than $t_0$, so the
geodesic $\exp (tX)$, $t\in [0,t_0]$ is not minimizing with respect to the
bi-invariant metric. The next lemma shows the converse statement is
also true.

\begin{lemma}\label{lemgeo}
Let $X$ be a unit vector and $\exp(tX)$,
$t\in [0,t_0]$, a geodesic from $e$ to $g=\exp(t_0 X)$
which
 is not minimizing with respect to the bi-invariant
 metric. Then for any $V\in\mathfrak{g}$ with $|V|_{bi}<1$, there is
 $t'\in [0,t_0)$ and a unit vector $X'\in \mathfrak{g}$, such that
 \begin{equation}\label{aaa}
 g=\exp(t_0 X)= \exp (t' X') \exp ((t_0-t')V).
 \end{equation}
\end{lemma}

\begin{proof}
We construct a sequence $t_n\in [0,t_0]$, and a sequence of unit
vectors $X_n\in \mathfrak{g}$ with respect to the bi-invariant metric,
inductively as follows. Let $t_1$ and $X_1$ be the pair such that
the geodesic $\exp (t X_1)$, $t\in [0,t_1]$, is minimizing from $e$
to $g$. Suppose  we have defined $t_i$ and $X_i$. Then we choose $t_{i+1}$ and
$X_{i+1}$ to be the pair such that the geodesic $\exp (t X_{i+1})$,
$t\in [0,t_{i+1}]$, is the shortest from $e$ to $g
\exp((t_i-t_0)V)$.  We have a sequence of equalities
$g=\exp(t_{i+1} X_{i+1})\exp((t_0-t_i)V)$, $\forall i\geq 0$, which imply that
$\exp(t_{i+1} X_{i+1})\exp(-t_i X_i)=\exp((t_{i}-t_{i-1})V)$. Using
the triangle inequality, we have $|t_{i+1}-t_i|\leq |V|_{\mathrm{bi}}
|t_i-t_{i-1}|$. So the sequence $t_n$ converges to some $t'\geq 0$.
Using  a suitable subsequence if necessary, we can assume that $X_n$  converge to some unit vector
$X'$ satisfying (\ref{aaa}). Since all the geodesics $\exp(t
X_n)$, $t\in [0,t_n]$, are minimizing, the limit geodesic $\exp (t
X')$, $t\in[0,t']$ is also a minimizing geodesic from $e$ to
$g\exp((t'-t_0)V)$. If $t'>t_0$, then the path given by the geodesic
from $e$ to $\exp((t'-t_0)V)$ and the geodesic from
$\exp((t'-t_0)V)$ to $g\exp((t'-t_0)V)=\exp(t_0 X)\exp((t'-t_0)V)$
has a length $t_0+|V|_{eq}(t'-t_0)<t'$, which is a
contradiction. On the other hand,    the condition
that $\exp(tX_0)$, $t\in [0,t_0]$, is not minimizing, implies that $t'\ne t_0$. Therefore $t'\in
[0,t_0)$. This completes the proof of the lemma.
\end{proof}

For any $g_1\neq g_2$, we can find $X\in\mathfrak{g}$ with
$|X|_{\mathrm{bi}}=1$, such that $(X,V)$ is the Killing vector field of
constant length $1$, and its flow curve, $\exp(tX)g_1\exp(-tV)$,
$t\in[0,t_0]$,  generates a minimizing geodesic from $g_1$ to $g_2$
for the Randers metric $F$. Denote the local one-parameter group of diffeomorphisms generated by
$(X,V)$ by $\phi_t$. Given $t'\in [0,t_0]$, if $\phi_{t'}$ is not
a Clifford-Wolf translation, then there is some $g'\in G$, such that
the geodesic $\exp(tX) g'$, $t\in [0,t']$,  is not minimizing. Then with respect to the
bi-invariant metric, the geodesic $\exp(tX)$, $t\in [0,t']$, does not minimizes the distance
 from $e$ to $\exp(t'X)$. By Lemma \ref{lemgeo}, there is
$X''$ with $|X''|_{\mathrm{bi}}=1$ and $t''\in[0,t')$, such that
$\exp(t'X)=\exp(t''X'')\exp((t'-t'')\mathrm{Ad}(g_1) V)$. This means that the
geodesic $\exp(tX'')g_1\exp(-tV)$, $t\in[0,t'']$,  gives a
path with length $t''$. It is shorter than the geodesic $\exp(tX)g_1\exp(-tV)$,
$t\in [0,t']$, from $g_1$ to $\exp(t' X)g_1\exp(-t'V)$. This is a
contradiction with the fact that for $t\in[0,t_0]$, $(X,V)$ generates a minimizing geodesic
from $g_1$ to $g_2$. Therefore $(G, F)$ is
CW-homogeneous. This completes  the
proof of Theorem \ref{thmmain}.


\section{Proof of Theorem \ref{preparation-theorem}}
\subsection{The theme of the proof}
Let $G$ be a compact connected simple Lie group with Lie algebra
$\mathfrak{g}$, $\mathfrak{h}$  a Cartan subalgebra of
$\mathfrak{g}$ with
$\dim_{\mathbb{R}}\mathfrak{h}=\mathrm{rk}(\mathfrak{g})=n$, and $W$  the corresponding
Weyl group. The bi-invariant metric on $G$ induces a $W$-invariant
linear metric on $\mathfrak{h}$. The condition in Theorem
\ref{preparation-theorem}, that the vector field generated by $X\in\mathfrak{g}$   is not a
Killing vector field of constant length of the  left
invariant Randers metric $F$ on $G$ when $F$  is not Riemannian,  is equivalent to the condition that
 any ellipsoid containing the $\mathrm{Ad}(G)$-orbit $X$
must be centered at $0$.

If $-\mbox{Id}$ is contained in the Weyl group, the proof goes
easily by applying the next lemma.

\begin{lemma}
If $-\mbox{Id}$ belongs to the Weyl group of a simple compact Lie
algebra $\mathfrak{g}$, then
any ellipsoid containing the $\mathrm{Ad}(G)$-orbit of a non-zero element $X\in\mathfrak{g}$ must be centered at $0$.
\end{lemma}

\begin{proof}
We only need to prove that for any non-zero $X$ in a Cartan
subalgebra, an ellipsoid containing its Weyl group orbit must be
centered at $0$. Denote the equation of the ellipsoid $E$ containing the
Weyl group orbit of $X\neq 0$  as
\begin{equation}
x^T A x + b^T x +c = 0.
\end{equation}
Then the ellipsoid defined by the equation
\begin{equation}
x^T A x - b^T x + c=0
\end{equation}
also contains the Weyl group orbit of $X$. So the Weyl group orbit
of $X$ is contained in the subspace $b^T x = 0$ if $b\neq 0$, or
equivalently, the ellipsoid $E$ is not centered at $0$. This conflicts
with the fact that the representation of the Weyl group is
irreducible on the Cartan subalgebra of a simple compact
$\mathfrak{g}$.
\end{proof}

It is well known that, the Weyl groups of all simple compact Lie algebras except $A_n$ with $n>1$, $D_n$ with
$n$ odd, and $E_6$,   contain the endomorphism
$-\mbox{Id}$ on the Cartan subalgebras. For the last three cases, we
need another lemma.

Let $\{v_1,\ldots,v_n\}$ be a  basis of the vector space $\mathfrak{h}$ given
by roots, with corresponding reflections $\rho_1,\ldots,\rho_n$ in
the Weyl group. Denote the linear subspace spanned  by $\rho_i$  as
$v_i^{\perp}$. Then we have

\begin{lemma}\label{observation-lemma}
Suppose for any $i=1,\ldots,n$, we can find $n$ points on the
Weyl group orbit of $X$ outside $v_i^\perp$, such that their orthogonal
projections in $v_i^{\perp}$ form an affine basis of $v_i^{\perp}$.
Then any ellipsoid containing the Weyl group orbit of $X$ must be
centered at the origin.
\end{lemma}

\begin{proof}
Let $E$ be the ellipsoid containing the Weyl group orbit of $X$. Suppose $l$ is
a line which has two intersectional points with $E$ and denote by $L$ the set of all the lines which is parallel to $l$. Define
$D$ to be the set of the middle points of the intersectional points of the lines in $L$ with $E$.
Then $D$  is
contained in a hyperplane. Notice that  the center of the ellipsoid must be contained in this
hyperplane. For each $i=1,\ldots,n$, the hyperplane containing all
middle points with the lines parallel to $v_i$ contains a set of $n$
projections from the orbit of $X$, which gives an affine basis. Thus
it is identical with $v_i^{\perp}$. Since the $v_i^{\perp}$'s have the
unique common point (the origin), the center of the ellipsoid  must  be the origin.
\end{proof}

The generic condition for $X$ are certain conditions for the
eigenvalue multiplicities for the cases of $A_n$ and $D_n$, and more
generally the isomorphic type of its isotropy group for $E_6$. They
are invariant under $\mathrm{Ad}$-actions, namely, if $X$ is generic, then each element of the $\mathrm{Ad}$-orbit of $X$ is also generic. If a generic element $X$ generates a Killing vector field
of constant length for some Randers metric $F$, then the restriction of $F$  to
each Cartan subalgebra must be Riemannian, since  the indicatrix
ellipsoid is centered at $0$. This implies that  $F$ must be Riemannian. Therefore we are left  to calculate the codimension of the complement of the
set of all generic elements in $\mathfrak{g}$. Next we deal with this problem
case by case.

\subsection{The case of $A_n$ with $n>1$}\label{case of An}
The diagonal matrices in $\mathfrak{su}(n+1)$ with $n>0$ form a Cartan
subalgebra. Any matrix in it can be identified with a vector
$X=(a_0,\ldots,a_n)\in\mathbb{R}^{n+1}$, in which
$\sum_{i=0}^{n}a_i=0$, and the eigenvalues are $a_i\sqrt{-1}$,
$i=0,\ldots,n$. The weyl group is the full permutation group $S_{n+1}$
for all the entries of $X$.

Take $v_i=(1,0,\ldots,0,-1,0,\ldots,0)$ with $-1$ in the
$i$th-entry. The reflection $\rho_i$ is the permutation which interchanges  $a_0$ and
$a_i$, and fixes all the other entries.

Let $X$ be a vector with the first entry equal to $a$, the last entry equal to $b$ such that $b<a$ and suppose all  the other entries
$c_i$ lies in $\mathbb{R}\backslash\{a,b\}$, $i=1,\ldots,n-1$. Then using the action of  $\rho_i$, $i=0,1,\ldots,n$,  we get the following $n$ points in the Weyl group
orbit of $X$:
\begin{eqnarray*}
  X_0&=& (a,b,c_1,\ldots,c_{n-1}),\\
   X_i&=&(a,c_i,c_1,
\ldots,c_{i-1},b, c_{i+1},\ldots,c_{n-1}),
\end{eqnarray*}
 for $i=1,\ldots,n-1$.
It is easy to see that neither of the above points is  contained in $v_1^{\perp}$.
To see that their orthogonal projections form an affine basis for
$v_1^\perp$, we only need to notice the fact that the set of  vectors
\begin{equation}
\mathrm{pr}(X_i)-pr(X_0)=((c_i-b)/2,0,\ldots,0,c_i-b,0,\ldots,0), \forall
i=1,\ldots,n-1,
\end{equation}
are linearly independent, which is obvious. The same technique can also be applied to
other $v_i^\perp$ with  $i>1$. Therefore  the condition
of Lemma \ref{observation-lemma} is satisfied. Hence the ellipsoid
containing this $X$ must be centered at $0$.  Thus the left-invariant
Randers metric with Killing vector fields of constant length given
by this $X$ must be Riemannian on each Cartan subalgebra.
Therefore $F$ is Riemannian.

Any matrix in $\mathfrak{su}(n+1)$ have $n+1$ imaginary eigenvalues, with their
multiplicities denoted as $\{n_1,\ldots,n_m\}$, $\sum_{i=1}^m
n_i=n+1$. The generic condition discussed above is in fact that
there are two $1$'s among all $n_i$'s. The matrix space $\mathfrak{su}(n+1)$ can
be naturally stratified as a finite union of subvarieties with
respect to the multiplicities. Notice that the isotropy group for
the ${\mathrm Ad}(SU(n))$-action at $X\in \mathfrak{su}(n+1)$ with eigenvalue
multiplicities $\{n_1,\ldots,n_m\}$ is $S(U(n_1)\times\cdots\times
U(n_m))$. So the set of the  matrices with the same eigenvalue multiplicities
$\{n_1,\ldots,n_m\}$ is a subvariety with codimension
\begin{equation}\label{codim-formula-1}
\sum_{i=1}^{m}n_i^2-m = \sum_{i=1}^m (n_i+1)(n_i-1)\geq n+m-2l,
\end{equation}
in which $l$ is the number of $1$'s among the $n_i$'s. Obviously,  when
$m\geq 3$ and $l\leq 1$, the codimension is $n+m-2l\geq n+1$, and when
$m=2$ and $l\leq 1$, it is also easy to see that the inequality in
(\ref{codim-formula-1}) is restrict. Thus each stratified subset in the
complement of the set of  the generic elements has a codimension at
least $n+1$, and  so does their union. This completes the proof of
Theorem \ref{preparation-theorem} for the case of $A_n$.

\subsection{The case of $D_n$ with odd $n>2$}
In a Cartan subalgebra of $\mathfrak{so}(2n)$ with $n>2$, the matrix can be
identified with a vector $X=(a_1,\ldots,a_n)\in\mathbb{R}^n$, such
that the eigenvalues of the matrix are $\pm a_i\sqrt{-1}$ for
$i=1,\ldots,n$. The Weyl group actions permute the entries and
change even numbers of signs of the entries arbitrarily. The
eigenvalue multiplicities for a matrix in $\mathfrak{so}(2n)$ can be denoted
as $\{n_0,n_1,\ldots,n_m\}$,  where $n_0$ is the even multiplicity
of the eigenvalue $0$, and the $n_i$'s are eigenvalue multiplicities
for positive multiples of $\sqrt{-1}$. Then we have $n_0+2\sum_{i=1}^m
n_i=2n$. Suppose  the eigenvalue
multiplicities $X\in\mathbb{R}^n$   satisfies the conditions that $m\geq 2$ and there is an $i>0$ such that $n_i=1$.
After a suitable change of the order and adjustment of  the signs, we can suppose that $X$ can be represented as
 $(a,b,c_1,\ldots,c_n)$ such that $a$ and $b$ are non-zero, and
$a$ has different absolute value from the others. If there is a $c_i$ such that $c_i=b$,
then we change the sign for that $c_i$ and $a$ simultaneously using a Weyl
group element. Therefore we can  assume further that $b\neq c_i$, for $i=1,\ldots,n-2$.
Let $v_i=(1,0,\ldots,-1,0,\ldots,0)$ be the root with a $-1$ in the
$(i+1)$-th entry, and denote $v_0=(1,1,0,\ldots,0)$. One can find the following $n$ points from the Weyl group
orbit of $X$ outside $v_1^\perp$:
 $(a,b,c_1,\ldots,c_{n-2})$,
$(a,c_i,c_1,\ldots,c_{i-1},b,c_{i+1},\ldots,c_{n-2})$ with
$i=1,\ldots,n-2$, and $(-a,-b,c_1,\ldots,c_{n-2})$. Neither of the $n$
points is  contained in $v_1^\perp$, and their orthogonal
projections in $v_1^\perp$ form an affine basis. This argument can also applied to
$v_0$ and the other $v_i$'s. Thus  $X$ is a
generic vector in the sense of Lemma \ref{observation-lemma}.

In this case the generic condition for $X$ can be stated as follows: the eigenvalue multiplicities
$\{n_0,\ldots,n_m\}$ of $X$ satisfy the condition that $m\geq 2$,  and there is a $n_i=1$
for some $i>0$. The isotropy group of the $\mathrm{Ad}(\mathrm{SO}(2n))$-action at $X$
with eigenvalue multiplicities $(n_0,\ldots,n_m)$ is isomorphic to
$\mathrm{SO}(n_0)\times \mathrm{U}(n_1)\times \mathrm{U}(n_m)$. Therefore   the subvariety of  the elements $X$ whose
eigenvalue multiplicities are all the same,  has  codimension
$n_0(n_0-1)/2+\sum_{i=1}^m n_i^2 - m$. If $m=1$ and $n_1=1$, then the
codimension equals $(n-1)(2n-3)>n+1$ when $n>2$. If for all $i>0$ we have  $n_i>1$, then the codimension is
\begin{eqnarray}
\frac{n_0(n_0-1)}{2}+\sum_{i=1}^m n_i^2 - m& = &
\frac{n_0(n_0-1)}{2} +
\sum_{i=1}^m (n_i-1)(n_i+1)\nonumber \\
& \geq & n_0-1 + \sum_{i=1}^m n_i + m \nonumber \\
& = & n+(m+\frac{n_0}{2}-1)\geq n
\end{eqnarray}
in which the equality holds only when $n_0/2+m=1$. However, one  can easily
check that the equality can not hold. So in the complement of the set of the generic elements,
 any stratified subvariety of the fixed type of eigenvalue
multiplicities has a codimension at least $n+1$. This completes the
proof of Theorem \ref{preparation-theorem} for the case of $D_n$.
%
%
%
%

\subsection{The case of $E_6$}

Now we consider the last case of  $E_6$, which is also the most difficult one.
The Cartan subalgebra of $E_6$ can be identified with
$\mathbb{R}^6$. The root system consists of $\alpha=(\pm 1,\pm
1,0,0,0,0)$ together with all permutations of $\alpha$  keeping the last entry $0$ fixed,
and all the vectors $(\pm 1/2,\pm 1/2, \pm 1/2, \pm 1/2$, $\pm 1/2,
\pm \sqrt{3}/2)$ with odd positive signs. It is easy to observe
that the set $\Pi$  of roots which are perpendicular to $(1/2,\ldots,1/2,-\sqrt{3}/2)$ consists of
 the permutations of $(1,-1,0,0,0,0)$,  together with the permutations of  the vectors $\pm(1/2,\ldots$,
$1/2,-1/2,\sqrt{3}/2)$, both keeping the last entry fixed. It is also easily checked that the set   set $\Pi$ forms
the root system of $A_5$. We can use an orthogonal automorphism of
$\mathbb{R}^6$ to map them to the standard root system of $A_5$,
i.e., all permutations of $(1,-1,0,0,0,0)$. More precisely, the
orthogonal automorphism keeps all permutations of $(1,-1,0,0,0,0)$
fixing the last entry $0$, and maps $(-1/2,\ldots,-1/2,1/2,-\sqrt{3}/2)$
to $(0,\ldots,0,1,-1)$. We can also assume that it maps the roots
$\pm(1/2,1/2,1/2,1/2,1/2, -\sqrt{3}/2)$ to
$\pm(3^{-1/2},\ldots,3^{-1/2})$, respectively. It is not hard to
check that it maps the remaining forty roots to $\pm((3^{-1/2}\pm
1)/2,\ldots,(3^{-1/2}\pm 1)/2)$ with three positive signs and three
negative signs among the six entries. To summarize, we have the
following lemma.

\begin{lemma}
The root system of $E_6$ can be represented as the union of the
standard root system of $A_5\times A_1$, i.e., all permutations of
$(1,-1,0,\ldots,0)$ and $\pm(3^{-1/2},\ldots,3^{-1/2})$, and all the vectors
$\pm((3^{-1/2}\pm 1)/2,\ldots,(3^{-1/2}\pm 1)/2)$ with three
positive signs and three negative signs among the six entries.
\end{lemma}

For any vector $X\in\mathbb{R}^6$, or the Cartan subalgebra of
$E_6$, the multiplicities of its different entries can be denoted as
a set of positive integers $\{n_1,\ldots,n_m\}$, with $\sum_{i=1}^m
n_m=6$. For simplicity, we will call it the multiplicity type of
$X$. Suppose  $X$ is a vector in the Cartan subalgebra whose
 all  entries sum to $0$, i.e.,   a vector in the Cartan
subalgebra of $A_5$. From Subsection \ref{case of An}, we see that  if the multiplicity type of $X$ contains two
$1$'s, then any ellipsoid in the Cartan subalgebra of $A_5$
containing the $A_5$-Weyl group orbit of $X$ is centered at $0$. In fact
it is true for more generic elements.

\begin{lemma}\label{51}
If $X$ has three different entries and the sum  of all its entries is equal to $0$,
then any ellipsoid in the Cartan subalgebra of $A_5$ which contains
the $A_5$-Weyl group orbit of $X$ must be centered at 0.
\end{lemma}

\begin{proof}
We only need to consider the case with multiplicity types
$\{1,2,3\}$ and $\{2,2,2\}$ and repeat the arguments in Subsection
\ref{case of An}.

Consider $X=(a,b,b,c,c,c)$ where $a$, $b$ and $c$ are three
different numbers with $a+2b+3c=0$. Let the roots
$v_i=(1,0,\ldots,0,-1,0,\ldots,0)$ with a $-1$ in the $(i+1)$-th
entry, $i=1,\ldots,5$. For $v_1=(1,-1,0,\ldots,0)$, the $A_5$-Weyl
group orbit of $X$ contains the points $(a,b,b,c,c,c)$, $(c,b,b,a,c,c)$,
$(c,b,b,c,a,c)$, $(c,b,b,c,c,a)$ and $(a,b,c,b,c,c)$. They are
vectors outside $v_1^\perp$ and their orthogonal projections form an
affine basis for $v_1^\perp$. For other $v_i$'s, the arguments are
similar. By Lemma \ref{observation-lemma}, any ellipsoid containing
the $A_5$-Weyl group orbit of $X$ is centered at $0$.

Consider $X=(a,a,b,b,c,c)$ where $a$, $b$ and $c$ are three
different numbers with $a+b+c=0$. Then for $v_1$, the $A_5$-Weyl
group orbit of $X$ contains the points $(a,b,a,b,c,c)$, $(a,b,b,a,c,c)$,
$(a,b,a,c,b,c)$, $(a,b,a,c,c,b)$ and $(a,c,a,b,c,b)$. They are
vectors outside $v_1^\perp$ and their projections form an affine
basis for $v_1^\perp$. Similar arguments  also apply  to other $v_i$'s. This  completes
 the proof of the lemma.
\end{proof}

Finally, we prove that   any ellipsoid containing
the $E_6$-Weyl group orbit of $X$ must be centered at $0$, provided that  $X$ has three different entries.

First we consider the case that the sum of all entries of $X$ is not
$0$. Let $E$ be any ellipsoid containing the $A_5\times A_1$-Weyl
group orbit of $X$. Notice that the Cartan subalgebra of the $A_1$
factor is generated by $v_0=(3^{-1/2},\ldots,3^{-1/2})$, and the
Cartan subalgebra of $A_5$ is the orthogonal complement $v_0^\perp$.
The $A_5\times A_1$-Weyl group orbit of $X$ has no intersection with
$v_0^{\perp}$, and Lemma \ref{51} implies that their orthogonal
projection to $v_0^\perp$ contains an affine basis for $v_0^\perp$.
So the middle points of the intersection between $E$ and the lines
parallel to $v_0$ are contained in $v_0^\perp$. Similarly, the
centers of the intersection ellipsoid between $E$ and the
hyperplanes parallel to $v_0^\perp$ is contained in $\mathbb{R}v_0$.
Since the center of $E$ must be contained in both $v_0^\perp$ and $\mathbb{R}v_0$, it must be $0$.

Next we consider the case that the sum of all entries of $X$ is $0$.
Let $E$ be any ellipsoid containing the $E_6$-Weyl group orbit of
$X$. Then Lemma \ref{51} implies that the intersection of  $E$ with
$v_0^\perp$ is an ellipsoid containing the $A_5$-Weyl group orbit of
$X$, which is centered at $0$. On the other hand,  any element $\rho$ in the
$E_6$-Weyl group maps $v_0^\perp$ to another subspace. Thus
the intersection between $E$ and $\rho(v_0^\perp)$ is also an
ellipsoid containing the $\rho$-images of the $A_5$-Weyl group orbit
of $X$, which is also centered at $0$. Therefore the center of $E$ must be $0$,
since it is contained in both the lines generated by $v_0$ and
$\rho(v_0)$ which intersect at $0$. This proves the assertion.

For non-zero vector $X$ with other multiplicity types, namely, $\{6\}$,
$\{5,1\}$, $\{4,2\}$ and $\{3,3\}$, if we can find a vector $X'$ in
its $E_6$-Weyl group orbit with three different entries, the same
statement is still true. If we assume further that any vector in its
Weyl group orbit does not have three different entries, then additional
linear relations will be required between the entries of $X$. So up
to nonzero constant scalars, we can only find finite $X$'s in the
Cartan subalgebra $\mathfrak{h}$, even if there does exist such $X$. The
centralizer of such a $X$ is the Lie algebra of the isotropy group
for the $\mathrm{Ad}(G)$-action. It  contains $\mathfrak{h}$ and is obviously not
abelian. Hence its dimension is at least $n+2$. So the union of all
the $\mathrm{Ad}(G)$-orbits of the generic elements in the above sense  has a codimension at least $n+1$ (
the dimension of the centralizers minus the one dimension for scalar
changes). Thus the assertion of Theorem \ref{preparation-theorem} holds for $E_6$. This completes the proof of Theorem \ref{preparation-theorem}.

\noindent {\bf Acknowledgements.}\quad  We would like to thank  Dr. Libing Huang and Dr. Zhiguang Hu for useful discussions. This work was finished during the second author's visit to the Chern institute of Mathematics. He would like to express his deep gratitude to the members of the institute for their hospitality.


\begin{thebibliography}{99}
\bibitem[AW76]{AW76} R. Azencott, E. Wilson, Homogeneous manifolds with negative curvature I, Trans. Amer. Math. Soc., 215 (1976), 323-362.

\bibitem[BCS00]{BCS00} D. Bao, S. S. Chern, Z. Shen, An Introduction to
Riemann-Finsler Geometry, Springer-Verlag, New York, 2000.

\bibitem[BN08-1]{BN081} V. N. Berestovskii, Yu. G. Nikonorov, Killing vector fields of constant length
on locally symmetric Riemannian manifolds, Transformation Groups, 13  (2008),
25每45.

\bibitem[BN08-2]{BN082} V. N. Berestovskii, Yu.G. Nikonorov, On $\delta$-homogeneous Riemannian manifolds, Diff. Geom. Appl., 26 (2008), 514每535.

\bibitem[BN09]{BN09}  V. N. Berestovskii, Yu.G. Nikonorov, Clifford-Wolf homogeneous Riemannian manifolds, Jour. Differ. Geom., 82 (2009), 467-500.

\bibitem[CS05]{CS04}  S. S. Chern, Z. Shen, Riemann-Finsler Geometry,
World Scientific Publishers, 2004.

\bibitem[DH02]{DH02} S. Deng and Z. Hou, The group of isometries of a Finsler space, Pacific J. Math, 207 (2002), 149-157.

 \bibitem[DMW86]{DMW86} I. Dotti-Miatello, R. Miatello and J. A. Wolf, Bounded isometries and homogeneous Riemannian quotient manifolds,  Geom. Dedicata,  21 (1986), 21每27.

 \bibitem[DE12]{DP} S. Deng, Clifford-Wolf translations of Finsler spaces of  negative flag curvature,
preprint.

\bibitem[DM12]{DM12} S. Deng and M. Xu, Clifford-Wolf translations of Finsler spaces, to appear in Forum Math., [Math.DG] arXiv:1201.3714v1.



\bibitem[DR83]{DR83} M. J. Druetta, Clifford translations in manifolds without focal points, Geom. Dedicata,  14 (1983),
95每103.

\bibitem[FR63]{FR63} H. Freudenthal, Clifford-Wolf-isometrien symmetrischer ra\"{u}me, Math. Ann., 150 (1963), 136-149.

\bibitem[HE74]{HE74} E. Heintze, On homogeneous manifolds of negative curvature, Math. Ann., 211 (1974), 23-34.

\bibitem[OT76]{OT76} T. Ochiai and T. Takahashi, The group of isometries of a left invariant metric on a Lie group, Math. Ann., 223 (1976), 91-96.

\bibitem[OZ69]{Oz69} V. Ozols, Critical points of the displacement function of an isometry, J. Differential Geometry,  3 (1969), 411每432.

\bibitem[OZ74]{Oz74} V. Ozols, Clifford translations of symmetric spaces, Proc. Amer. Math. Soc.,  44 (1974),
169每175.



\bibitem[SH01]{SH2} Z. Shen, Differential Geometry of Sprays and Finsler Spaces, Kluwer, Dordrent, 2001.

 \bibitem[WO60]{WO60} J. A. Wolf, Sur la classification des varietes riemanniennes homogenes a courbure constante, C. R. Math. Acad. Sci. Paris,  250 (1960),
3443-3445.

\bibitem[WO62]{WO62} J. A. Wolf, Locally symmetric homogeneous spaces, Commentarii Mathematici Helvetici,  37 (1962/63),
65-101.

\bibitem[WO64]{WO64} J. A. Wolf, Homogeneity and bounded isometries in manifolds of negative curvature, Illinois J. Math., 8 (1964), 14-18.

\end{thebibliography}
\end{document}